\documentclass[11pt]{article}
\usepackage[margin=3.5cm,nohead]{geometry}

\usepackage{dsfont}

\usepackage{psfrag}

\usepackage{calrsfs}
\usepackage{mathrsfs}
\usepackage{pdfsync}
\usepackage{amsmath, amsthm, amssymb, amsfonts}
\usepackage[shortlabels]{enumitem}

\usepackage{url}
\usepackage{float}

\usepackage[usenames]{color}

\usepackage{tikz}

\usetikzlibrary{arrows,decorations.pathmorphing,backgrounds,positioning,fit,automata}

\usetikzlibrary{arrows}
\usetikzlibrary{petri}
\usetikzlibrary{topaths}

\usepackage{indentfirst,calc,euscript}
\usepackage{setspace}
\usepackage[reals]{layout}
\usepackage{xr}
\usepackage{amscd}

\usepackage{sgame}
\usepackage{subfigure}

\usepackage[colorlinks=true,breaklinks=true,bookmarks=true,urlcolor=blue,
     citecolor=blue,linkcolor=blue,bookmarksopen=false,draft=false]{hyperref}

\newcommand{\ignore}[1]{}

\newtheorem{theorem}{Theorem}

\newtheorem{proposition}[theorem]{Proposition}

\newtheorem{property}{Property}

\theoremstyle{definition}

\newtheorem{definition}[theorem]{Definition}

\numberwithin{equation}{section}
\numberwithin{theorem}{section}

\newcommand{\m}{\mathbb}

\DeclareMathOperator*{\cupp}{\cup}

\author{Guillaume Garnier\thanks{Laboratoire Jacques-Louis Lions, Sorbonne Université, INRIA, 4 place Jussieu, Paris, France.
}, Bruno Ziliotto\thanks{CEREMADE, CNRS, Universit\'e Paris Dauphine, PSL Research Institute, Paris, France.
}}

\title{Percolation games}
\date{}
\begin{document}
\maketitle
\bibliographystyle{plain}
\begin{abstract}
This paper introduces a discrete-time stochastic game class on $\m{Z}^d$, which plays the role of a toy model for the well-known problem of stochastic homogenization of Hamilton-Jacobi equations. Conditions are provided under which the $n$-stage game value converges as $n$ tends to infinity, and connections with homogenization theory is discussed. 
\end{abstract}
\section{Introduction}
Shapley \cite{SH53} has introduced the seminal model of \textit{zero-sum stochastic games}. It involves two players playing repeatedly a zero-sum game that depends on a variable called \textit{state}. 
The state evolves stochastically from one stage to the other, and its law at stage $m+1$ is determined by the state and the actions at stage $m$. Players know the current state before playing, and observe each other's past actions. Action and state sets are finite. In the \textit{$n$-stage game}, the goal of Player 1 is to maximize the expected average payoff over $n$ stages, and Player 2's goal is to minimize it. The $n$-stage game has a \textit{value}, which is the maxmin (or the minmax) of the $n$-stage payoff over players' strategy sets. The value is denoted by $v_n$, and can be interpreted as the payoff solution of the game: when players play rationally, Player 1 should get $v_n$, while Player 2 should get the opposite. The study of stochastic games with \textit{long duration} has been a very active field of research in the past decades. Bewley and Kohlberg \cite{BK76} have proved that $(v_n)$ converges as $n$ tends to infinity. The \textit{limit value} can hence be interpreted as an approximate payoff solution of the game with long duration. Following this seminal work, limit value has been proven to exist in several variations of this model, including stochastic games with unobserved states or actions (see \cite{JN16,SZ16} for surveys), while several counter-examples have been pointed out in stochastic games with finite action sets or state space \cite{vigeral13,Z13}. Related questions such as the dependence of optimal strategies with respect to $n$ (\textit{uniform approach} \cite{MN81}) or the characterization and computability of the limit value \cite{AOB18,OB21,HKLMT11} have inspired fruitful lines of research. 

Recently, 
stochastic games have been used to handle problems arising in Partial Differential equations \cite{KS06,KS10,IS11,Z15}, such as Hamilton-Jacobi or Laplace equations \cite{AS12,PSSW09,PS08,LPS13}. More specifically, this paper is motivated by the problem of 
\textit{stochastic homogenization}  of Hamilton-Jacobi equations. The latter is a class of PDEs that is widely used to study Hamiltonian systems arising in mechanics. A central research area is to study limit properties of solutions when the time scale and space scale become larger and larger: when solutions converge in some sense, we say that the PDE system \textit{homogenizes}. Lions, Papanicolaou and Varadhan \cite{LPV86} proved that Hamilton-Jacobi equations with a space-periodic Hamiltonian do homogenize. 
The extension of these results to the case of random Hamiltonians with a stationary ergodic law (\textit{stochastic homogenization}) has motivated a vast literature in the last 30 years, featuring a variety of mathematical tools such as weak KAM theory, first-passage percolation, or large deviations. Though huge progress has been achieved for convex Hamiltonians, the non-convex case is far from being understood. Its importance is both theoretical and practical, since non-convexity 
gives rise to challenging mathematical phenomena (e.g. \textit{envelope shocks} \cite{evans14} or \textit{lack of sub-additivity} \cite{AC15}), and appears  in many applications, including \textit{equations of forced mean curvature motion} that describe in particular fire front propagation \cite{peters01,williams85}.

Hamilton-Jacobi equations can be associated with a class of continuous-time two-player zero-sum games called \textit{differential games}: the value of such differential games satisfy Hamilton-Jacobi equations, and conversely, for each ``standard" Hamilton-Jacobi equation, there exists a differential game which value satisfies this equation \cite{ES84}. Moreover, the 1-Player case corresponds to convex Hamiltonian, and the 2-Player case corresponds to non-convex Hamiltonians. Furthermore, homogenization is related to the existence of limit value in such games. 
This connection has been used by the second author, who provided examples of non-convex Hamilton-Jacobi equation that do not homogenize in a stationary ergodic environment \cite{Z17,FFZ21}. 
The examples are themselves inspired from a stochastic game with state space $\m{Z}^2$, that does not have a limit value. Despite these examples, the general question of which non-convex equation homogenize remains largely uncharted. 

The main goal of this paper is to provide a discrete-time stochastic game framework that is well-adapted to the study of non-convex stochastic homogenization, in the sense that understanding existence of limit value in such games is likely to infer new results in the related PDE problem. Such an approach has several advantages. First, putting the problem in a discrete-time game framework allows to use in an elementary way the \textit{strategic approach}, that is, putting oneself in players' shoes and guessing what is the best decision to take. More generally, this enables to use the abundant literature on existence of limit value in discrete-time games. Last, this paper will make the homogenization problem more accessible to game-theorists, and bring to PDE theorists a new perspective on the problem.

More specifically, this paper presents a new stochastic game model called \textit{percolation games}. It features a state space $\m{Z}^d$, $d \geq 1$, and random payoffs, with stationary ergodic distribution across $\m{Z}^d$. Limit value is proved to exist under the assumptions that payoff distribution is i.i.d. across $\m{Z}^d$, and the game is ``oriented'', meaning that the state always moves in the same direction, irrespective of players' actions. An example of an oriented percolation game without a limit value is then presented. Furthermore, connections with homogenization are discussed, and several promising research directions are presented. 
%
\section{Model and results}
A \textit{percolation game} is a tuple $\Gamma=(I,J,\Omega,g,q)$, where:
\begin{itemize}
\item
$I$ and $J$ are finite sets representing respectively Player 1's action set and Player 2's action set,
\item
$\mathcal{E}=(\Omega,\mathcal{F},\m{P})$  is a probability space, equipped with a family of measurable transformations $\tau_z: \Omega \rightarrow \Omega, z \in \m{Z}^d$, satisfying for all $z, z' \in \m{Z}^d$ and $F \in \mathcal{F}$
\begin{equation*}
\m{P}(\tau_z(F))=\m{P}(F) \quad \text{and} \quad \tau_{z+z'}=\tau_z \circ \tau_{z'}.
\end{equation*}
Moreover, $\tau$ is \textit{ergodic}: if $F \in \mathcal{F}$ is \textit{invariant}, that is, for all $z \in \m{Z}^d$, $\tau_z(F)=F$, then $\m{P}(F) \in \left\{0,1\right\}$. 
\item
$g=(\omega \rightarrow g_{\omega}(z,i,j)), \ (z,i,j) \in \m{Z}^d \times I \times J$ is a collection of real-valued random variables defined on $\Omega$, uniformly bounded and \textit{stationary}, that is:
for all $(z,z',i,j,\omega)$,
\begin{equation*}
g_{\tau_z(\omega)}(z',i,j)=g_{\omega}(z+z',i,j).
\end{equation*}
 The last two assumptions can be summarized by saying that the law of $g$ is invariant by space translation (stationarity), and any event that is invariant by space translation has either probability 0 or 1 (ergodicity). When the random variables $\left\{\omega \rightarrow (g_\omega(z,i,j))_{(i,j) \in I \times J}\right\}_{z\in \m{Z}^d}$ are i.i.d., these assumptions are satisfied. 
\item
$q$ is the \textit{transition function}, and is a mapping from $I \times J$ to $\m{Z}^d$.
\end{itemize}
Given some initial state $z \in \m{Z}^d$ and $\omega \in \Omega$, the game proceeds as follows: 
\begin{itemize}
\item
Players are informed of $\omega$,
\item
At each stage $m \geq 1$, Player 1 chooses an action $i_m$, and then knowing Player 1's action, Player 2 chooses an action $j_m$. Players 1 and 2 receive respectively payoffs $g_{\omega}(z_m,i_m,j_m)$ and $-g_{\omega}(z_m,i_m,j_m)$. Then,
the new state is $z_{m+1}:=z_m+q(i_m,j_m)$. 
\end{itemize}
A \textit{strategy} for Player 1 is a collection of mappings $\sigma_m: (\m{Z}^d \times I \times J)^{m-1} \times \m{Z}^d \rightarrow I, m \geq 1$, and a strategy for Player 2 is 
a collection of mappings $\tau_m: (\m{Z}^d \times I \times J)^{m-1} \times \m{Z}^d \times I \rightarrow J, m \geq 1$. The set of strategies of Player 1 (resp., Player 2) is denoted by $\Sigma$ (resp., $T$). 
An initial state $z \in \m{Z}^d$ and a pair $(\sigma,\tau)$ induce a sequence $(z_m,i_m,j_m)_{m \geq 1}$, where $z_m$ is the state at stage $m$, $i_m$ is the action for Player 1, and $j_m$ is the action for Player 2. 
The \textit{$n$-stage game} $\Gamma^\omega_n(z)$ is the zero-sum game with strategy sets $\Sigma$ and $T$, and payoff function $\gamma^{\omega}_n:\Sigma \times T \rightarrow \m{R}$ defined by
\begin{equation*}
\gamma^{\omega}_n(z,\sigma,\tau):=\frac{1}{n} \sum_{m=1}^n g_\omega(z_m,i_m,j_m).
\end{equation*}
The \textit{value} of $\Gamma^\omega_n(z)$ is the real number $v^{\omega}_n(z)$ such that
\begin{equation*}
v^{\omega}_n(z):=\max_{\sigma \in \Sigma} \min_{\tau \in T} \gamma^{\omega}_n(z,\sigma,\tau)= \min_{\tau \in T} \max_{\sigma \in \Sigma} \gamma^{\omega}_n(z,\sigma,\tau).
\end{equation*}
In the sequel, $v_n(z)$ designates the random variable $\omega \rightarrow v_n^{\omega}(z)$.
We now introduce key assumptions behind our main result: 
\begin{definition}
A game is \textit{oriented} if there exists $u \in \m{Z}^d$ such that for all $(i,j) \in I \times J$, $q(i,j)\cdot u>0$. 
\end{definition}
In an oriented game, whatever be players actions, the coordinate of the state with respect to $u$ always increases. The increasing rate may depend on players actions. 
\begin{definition}
A percolation game is i.i.d. if the random variables $\left\{\omega \rightarrow (g_\omega(z,i,j))_{(i,j) \in I \times J}\right\}_{z\in \m{Z}^d}$ are i.i.d.
\end{definition}

We are now ready to state our main result.
\begin{theorem} \label{theo:oriented}
Consider a percolation game that is i.i.d. and oriented. For all $z \in \m{Z}^d$, $v_n(z)$ converges $\m{P}$-a.s. to a constant $v_\infty \in \m{R}$, as $n$ tends to infinity. Moreover, there exists $A, B>0$ such that for all $n \geq 1$, for all $z \in \m{Z}^d$ and $\lambda \geq 0$,
\begin{equation} \label{eq:conc}
\m{P}\left(|v_n(z)-v_\infty| \geq \lambda+A \ln(n+1)^{1/2} n^{-1/2}\right) \leq \exp\left(-B \lambda^2 n \right). 
\end{equation}
In particular, $\m{E}(v_n)$ converges to $v_\infty$ at a rate $O(\ln(n)n^{-1/2})$. 
\end{theorem}
\begin{theorem} \label{theo:ce}
There exists an oriented percolation game in $\m{Z}^3$, with payoffs in $\left\{0,1\right\}$, such that for all $z \in \m{Z}^3$, $\m{P}$-almost surely \emph{($\m{P}$-a.s.)}, 
\begin{equation*}
\limsup_{n \rightarrow +\infty} v_n(z)=1 \quad \text{and} \quad \liminf_{n \rightarrow +\infty} v_n(z)=0.
\end{equation*}
\end{theorem}
\section{Proofs}
\subsection{Proof of Theorem \ref{theo:oriented}}
Denote $\left\|g\right\|_\infty:=\sup_{(z,i,j) \in \m{Z}^d \times I \times J} |g(z,i,j)|$ and $\left\|q\right\|_\infty:=\sup_{(i,j,k) \in I \times J \times [1..d]} \left|[q(i,j)]_k \right|$. Let $n \geq 1$ and $H:=\left\{z \in \m{Z}^d \ | \ z \cdot u =0 \right\}$. There exists a constant $C$ depending only on $u$ and $d$, such that the discrete cube $\mathcal{C}:=[0.. n \cdot \left\|q\right\|_\infty]^d$ can be covered by
a collection $(H_r)_{1 \leq r \leq Cn}$ of discrete hyperplanes that are translations of $H$: for each $r \in [1..Cn]$, the set $H_r$ can be written as
$H_r=\left\{z \in \m{Z}^d| z \cdot u=h_r\right\}$, for some $h_r \in \m{Z}$, and 
$\mathcal{C} \subset \cup_{r=1}^{Cn} H_r$. For $r \in [1..Cn]$, let $\mathcal{F}_r$ be the $\sigma$-algebra generated by the random variables $\omega \rightarrow g_\omega(z,i,j), \ (z,i,j) \in\cupp_{r'=1}^{r} H_{r'} \times I \times J$, and let $\mathcal{F}_0$ be the trivial $\sigma$-algebra. 
We claim that for $r \in [0..Cn-1]$,
\begin{equation} \label{eq_increment}
|\m{E}(v_n(0)|\mathcal{F}_{r+1})-\m{E}(v_n(0)|\mathcal{F}_r)| \leq \frac{2\left\|g\right\|_\infty}{n} \quad \m{P}-\text{a.s.},
\end{equation}
where 0 stands for the origin of $\m{Z}^d$. 
%
Indeed, define an auxiliary game $\Gamma'$ with the same action sets and transition function, and with payoff function defined for all $(\omega,i,j) \in \Omega \times I \times J$ by $g_\omega'(z,i,j)=0$ when $z \in H_{r+1}$, and $g_\omega'(z,i,j)=g_\omega(z,i,j)$ otherwise. For $\omega \in \Omega$, let $\gamma'^\omega_n$ be the $n$-stage payoff, and $v'^\omega_n$ the $n$-stage value. 
Because $\Gamma$ is oriented, irrespective of players' strategies, there can be at most one stage where the state lies in $H_{r+1}$. It follows that for all $(\sigma,\tau) \in \Sigma \times T$,
\begin{eqnarray*}
|\gamma'^{\omega}_n(0,\sigma,\tau)-\gamma^{\omega}_n(0,\sigma,\tau)| \leq \frac{\left\|g\right\|_\infty}{n},
\end{eqnarray*}
hence $\left|v'_n(0)-v_n(0)\right| \leq \left\|g\right\|_{\infty}$ $\m{P}$-a.s.
Moreover, because $\Gamma$ is i.i.d., $(v'_n(0),\mathcal{F}_r)$ is independent of random variables $\left\{\omega \rightarrow (g_\omega(z,i,j))_{(i,j) \in I \times J}, \ z \in H_{r+1} \right\}$. It follows that
$\m{E}(v'_n(0)|\mathcal{F}_{r+1})=\m{E}(v'_n(0)|\mathcal{F}_{r}) \ \m{P}\text{-a.s.}$, hence
\begin{eqnarray*}
|\m{E}(v_n(0)|\mathcal{F}_{r+1})-\m{E}(v_n(0)|\mathcal{F}_r)| &\leq& \frac{2\left\|g\right\|_\infty}{n} \quad \m{P}-\text{a.s.} 
\end{eqnarray*}
Let $W_r:=\m{E}(v_n(0)|\mathcal{F}_r), r \in [0..Cn]$. The process $(W_r)_{r \in [0..An]}$ is a martingale, and $W_0=\m{E}(v_n(0))$. By the stationarity assumption, the expectation of $v_n(z)$ is independent of $z$, thus we will write $\m{E}(v_n)$ instead of $\m{E}(v_n(0))$. Moreover, in the $n$-stage game starting from 0, the state stays in 
$\mathcal{C}$, hence
$v_n(0)$ is $\mathcal{F}_{Cn}$-measurable: $W_{Cn}=v_n(0)$.
 Consequently, by Azuma's inequality and \eqref{eq_increment}, for all $\lambda \geq 0$,
\begin{equation} \label{ineq_azuma}
\m{P}(|v_n(0)-\m{E}(v_n)| \geq \lambda) \leq \exp\left(\frac{-\lambda^2 n}{8\left\|g\right\|_\infty^2}\right).
\end{equation}
Using the union bound, and then the fact that variables $v_n(z)$ and $v_n(0)$ have the same distribution by stationarity, 
\begin{align*}
\m{P}(\exists z \in \mathcal{C}, |v_n(z)-\m{E}(v_n)| \geq \lambda) &\leq \sum_{z \in \mathcal{C}}  \m{P}(|v_n(z)-\m{E}(v_n)| \geq \lambda) 
\\
&=  \sum_{z \in \mathcal{C}}  
 \m{P}(|v_n(0)-\m{E}(v_n)| \geq \lambda)  
\\
&\leq (2 n \cdot \left\|q\right\|_\infty +1)^d \exp\left(\frac{-\lambda^2 n}{8\left\|g\right\|_\infty^2}\right).
\end{align*}
Taking $\lambda_n:=2 \sqrt{2} \ln(2n \cdot \left\|q\right\|_\infty+1)^{1/2}  (d+1)\left\|g\right\|_\infty n^{-1/2}$, we get 
\begin{eqnarray*}
\m{P}(\exists z \in \mathcal{C}, |v_n(z)-\m{E}(v_n)| \geq \lambda_n)
&\leq& (2 n \cdot \left\|q\right\|_\infty +1)^{-1} \leq n^{-1} ,
\end{eqnarray*}
hence
\begin{equation*}
\m{P}\left(\min_{z \in \mathcal{C}}{v_n(z)} \leq  \m{E}(v_n)-\lambda_n \right) \leq n^{-1}. 
\end{equation*}
We deduce that
\begin{eqnarray*}
\m{E}\left(\min_{z \in \mathcal{C}}{v_n(z)}\right) &\geq& -\m{P}\left(\min_{z \in \mathcal{C}}{v_n(z)} \leq  \m{E}(v_n)-\lambda_n \right) \left\|g\right\|_\infty
\\
&+&\m{P}\left(\min_{z \in \mathcal{C}}{v_n(z)} \geq  \m{E}(v_n)-\lambda_n \right)(\m{E}(v_n)-\lambda_n)
\\
&\geq& -n^{-1} \left\|g\right\|_\infty +\m{E}(v_n)-\lambda_n-n^{-1} |\m{E}(v_n)-\lambda_n|
\end{eqnarray*}
Hence, there exists $D>0$ such that for all $n \geq 1$,
\begin{equation*}
\m{E}\left(\min_{z \in \mathcal{C}}{v_n(z)}\right) \geq \m{E}(v_n) -D \ln(n+1)^{1/2} n^{-1/2}.
\end{equation*}
Let $n \geq 1$ and $m \in [1..n]$. 
Consider the $(m+n)$-stage game starting from 0, and the following strategy for Player 1: play optimally in the $m$-stage game, then again optimally in the $n$-stage game starting from state $z_{m+1}$. At stage $m+1$, the state is in $\mathcal{C}$, hence this strategy guarantees $\frac{m}{m+n} v_m(0)+\frac{n}{m+n} \min_{z \in \mathcal{C}}{v_n(z)}$ $\m{P}$-a.s., and
\begin{eqnarray}
\m{E}(v_{m+n}) &\geq& \frac{m}{m+n} \m{E}(v_m)+\frac{n}{m+n} \m{E}\left(\min_{z \in \mathcal{C}}{v_n(z)}\right) 
\nonumber \\ \label{eq:sub}
&\geq&
\frac{m}{m+n} \m{E}(v_m)+\frac{n}{m+n}\left[\m{E}(v_n) -D \ln(n+1)^{1/2} n^{-1/2}\right].
\end{eqnarray}
For all $n \geq 1$, set $a_n:=-n \m{E}(v_n)$, and $f:=D \ln(n+1)^{1/2} n^{1/2}$. By the previous inequality, the sequence $(a_n)$ is \textit{subadditive with an error term $f$} \cite{BE52}: 
$f$ is non-negative and non-decreasing, and for all $m,n \geq 1$, $a_{m+n} \leq a_m+a_n+f(m+n)$. Moreover, 
$\sum_{n \geq 1} f(n)/n^2<+\infty$. By \cite{BE52}, $a_n/n$ converges, hence $\m{E}(v_n)$ converges. Call $v_\infty$ the limit. To obtain the rate of convergence,
inequality \eqref{eq:sub} gives by induction that for all $\ell \geq 1$ and $n \geq 1$,
\begin{eqnarray*} 
\m{E}(v_{2^{\ell} n}) &\geq& 
\m{E}(v_n)-D\sum_{k=0}^{\ell-1} \ln(2^{k}n+1)^{1/2} 2^{-k/2} n^{-1/2}
\\
&\geq& \m{E}(v_n)-A \ln(n+1)^{1/2} n^{-1/2},
\end{eqnarray*}
where $A$ is a universal constant. 


Making $\ell$ tends to infinity, this yields for all $n \geq 1$,
\begin{eqnarray*}
v_\infty \geq \m{E}(v_n)-A \ln(n+1)^{1/2} n^{-1/2}.
\end{eqnarray*}
Switching Players 1 and 2's roles, we obtain that for all $n \geq 1$,
\begin{equation*}
\left|v_\infty-\m{E}(v_n)\right| \leq A \ln(n+1)^{1/2} n^{-1/2}.
\end{equation*}
Let us conclude on the proof of Theorem \ref{theo:oriented}.
Let $\lambda \geq 0$. Combining the above inequality with \eqref{ineq_azuma}, we get that for all $n \geq 1$,
\begin{eqnarray*}
\m{P}(|v_n(0)-v_\infty| \geq \lambda+ A \ln(n+1)^{1/2} n^{-1/2}) &\leq&
 \exp\left(\frac{-\lambda^2 n}{8\left\|g\right\|_\infty^2}\right).
 \end{eqnarray*}
 Thus, inequality \eqref{eq:conc} holds for $B:=(8\left\|g\right\|_\infty^2)^{-1}$. Last, this inequality implies that $(v_n(0))$ converges $\m{P}$-almost surely, 
 and by stationarity, similar properties hold for $v_n(z)$, $z \in \m{Z}^d$. 

 \subsection{Proof of Theorem \ref{theo:ce}}
 Consider a percolation game on $\m{Z}^3$, with action sets $I=J=\left\{-1,0,1\right\}$, and transition function $q(i,j)=(i,j,1)$. Hence, Player 1 can move the state left or right or keep it at the same abscissa, and Player 2 can move the state backward or forward, or keep it at the same ordinate. Last, independently of players'actions, the state moves one edge up at each stage. This is an oriented game, with direction $u=(0,0,1)$. 
%
We now define the payoff function.  

Let $(T_k)$ be the sequence defined for $k \geq 1$ by $T_k=2^k$. 
Let $(X^j_{k,z})_{(j,k,z) \in \left\{1,2\right\} \times \m{N}^*\times \m{Z}^3}$ be a sequence of independent random variables defined on a probability space $(\Omega,\mathcal{F},\m{P})$, such that for all $(j,k,z) \in \left\{1,2\right\} \times \m{N}^* \times \m{Z}^3$, $X^j_{k,z}$ follows a Bernoulli of parameter $T _k^{-3}$. 
 Payoffs depend only on the state variable, and for each $\omega \in \Omega$, the payoff function $g_{\omega}:\m{Z}^3 \rightarrow \left\{0,1\right\}$ is built in two phases.
\\

\textbf{Phase 1}
\\

First, a mapping $g^1_{\omega}:\m{Z}^3 \rightarrow \left\{0,1\right\}$ is built through the following step-by-step procedure:
\begin{itemize}
\item
Step $k=0$: take $g^1_{\omega}:=0$ as the initial distribution of payoffs. 
\item
Step $k \geq 1$: for each $z \in \m{Z}^3$ such that $X^{1}_{k,z}(\omega)=1$, consider the square centered on $z$, with side length $T_k$, orthogonal to the vector $(1,0,0)$, which shall be called \textit{1-square of length $T_k$}. For each $z' \in \m{Z}^3$ that lies in the square, set $g^1_{\omega}(z'):=1$ (note that $g^1_{\omega}(z')$ may have already been defined as being 1 at an earlier step).
\end{itemize}
At the end of Phase 1, we have a map $g^1_{\omega}: \m{Z}^3 \rightarrow \left\{0,1\right\}$. Then, go to Phase 2:

\vspace{0.5cm}
\textbf{Phase 2}
\vspace{0.5cm}
\\
The mapping $g_{\omega}:\m{Z}^3 \rightarrow \left\{0,1\right\}$ is built through the following step-by-step procedure:
\begin{itemize}
\item
Step $k=0$: take $g_{\omega}:=g^1_{\omega}$ as the initial distribution of payoffs.
\item
Step $k \geq 1$: for each $z \in \m{Z}^3$ such that $X^{2}_{k,z}(\omega)=1$, consider the square centered on $z$, with side length $T_k$, and orthogonal to the vector $(0,1,0)$, which shall be called \textit{0-square of length $T_k$}. For each $z' \in \m{Z}^3$ that lies in the square, proceed as follows:
\begin{itemize}
\item
If $z'$ lies in a 1-square of size $T_{k'}$ with $k' \geq k$, then $g_{\omega}(z')$ is not modified. 
\item
Otherwise, set $g_{\omega}(z'):=0$. 
\end{itemize}
\end{itemize}
\newpage
The figure below represents 1-squares in green and 0-squares in red:
\begin{center}
\includegraphics[scale=0.3]{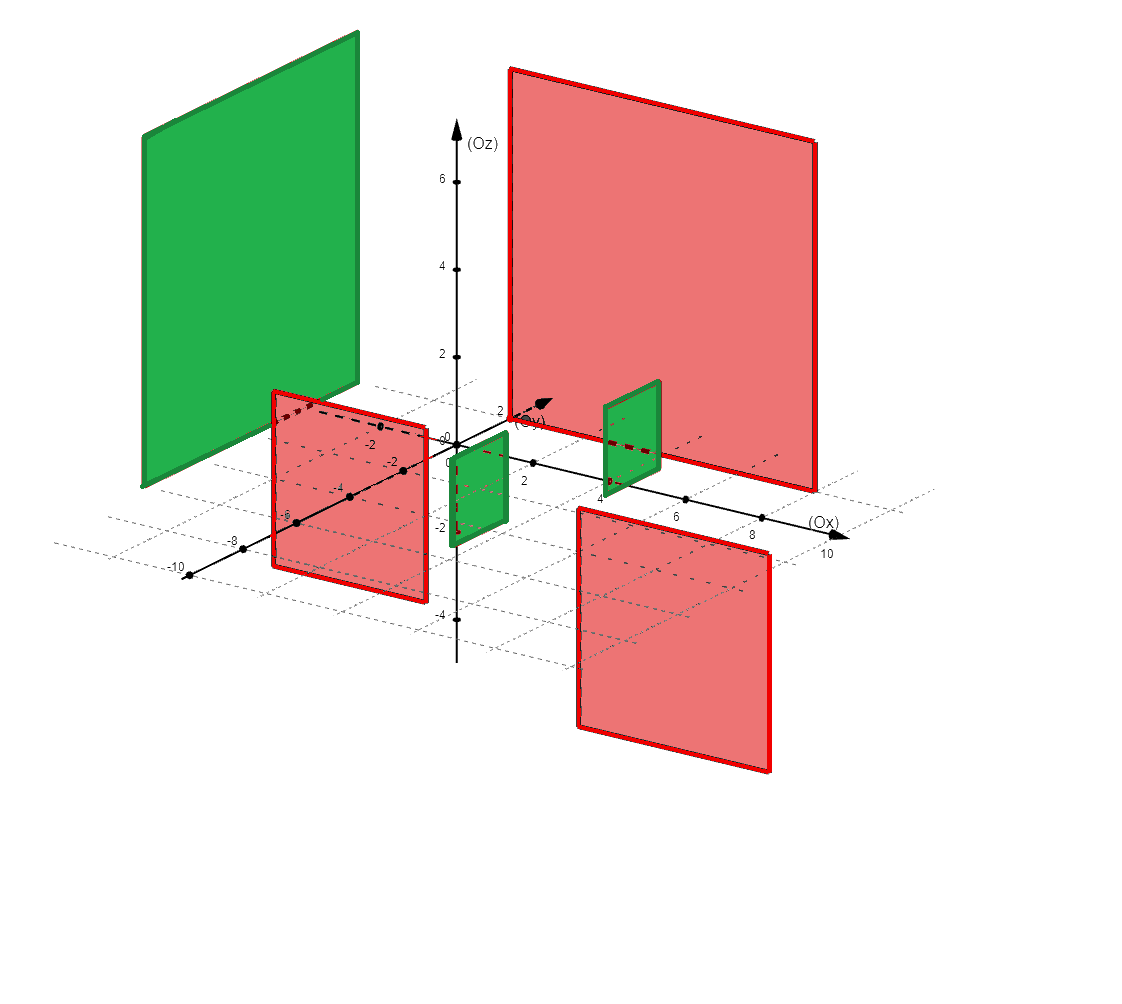}
\end{center}
A key feature of Phase 2 is that, whenever a 0-square intersects a 1-square, the intersection ``turns 1'' if the 0-square's length is smaller than the one of the 1-square, and ``turns 0'' otherwise. 

Moreover, the law of $g_\omega$ is stationary and ergodic.  
Indeed, let $A$ be an invariant event. For all $n \in \m{N^*}$, we denote by $\mathcal{F}_n$ the $\sigma$-algebra generated by the random variable 
\begin{equation*}
    (X^j_{k,z})_{(j,k,z) \in \left\{1,2\right\} \times [1..n] \times [-n..n]^3}\, .
\end{equation*}
For all $\varepsilon > 0$, there exists $n \in \m{N}^*$ and an event $A_n \subset A$ such that $A_n$ is $\mathcal{F}_n$-measurable and $\m{P}(A_n) \geq \m{P}(A) - \varepsilon$.\\
We define $A_n'$ as the translation of $A_n$ with respect to the vector $(0, 0, n+1+2^{n-1})$. \\
By hypothesis $A$ is $\m{Z}^3$--invariant, hence $A_n' \subset A$. Furthermore we have $\m{P}(A_n) = \m{P}(A_n')$, because the law of $g_\omega$ is stationary. \\
By construction, the events $A_n$ and $A_n'$ are independent. This implies
\begin{subequations}
    \begin{align*}
        \m{P}(A \cap A) &\leq \m{P}(A_n \cap A_n') + 2 \varepsilon \\
            &= \m{P}(A_n) \m{P}(A_n') + 2 \varepsilon \\
            &\leq \m{P}(A)^2 + 5\varepsilon.
    \end{align*}
    \end{subequations}
Since $\varepsilon$ is chosen arbitrarily, we have $\m{P}(A) \leq \m{P}(A)^2$.
We conclude that $$\m{P}(A) \in \{0, 1\}.$$

A 1-square is \textit{complete} if all its elements $z \in \m{Z}^3$ satisfy $g_{\omega}(z)=1$: in other words, it is not intersected by a larger 0-square. In the same vein, a 0-square is \textit{complete} if all its elements $z$ satisfy $g_{\omega}(z)=0$.

Equip $\m{Z}^3$ with the 1-norm. 
The construction of $(g_{\omega})_{\omega \in \Omega}$ has been made such that for all $\varepsilon>0$, $\m{P}$-almost surely, the following properties hold:
\begin{property} \label{prop:good_scenario}
There exists a sequence $(n_k)$ going to infinity such that for all $k \geq 1$, there exists a complete 1-square of length $4 T_{n_k}$ at a distance smaller or equal to $\varepsilon T_{n_k}$ from the origin. 
\end{property}
\begin{property} \label{prop:bad_scenario}
There exists a sequence $(n'_k)$ going to infinity such that for all $k \geq 1$, there exists a complete 0-square of length $4 T_{n_k}$ at a distance smaller or equal to $\varepsilon T_{n'_k}$ from the origin.
\end{property}
\begin{proof}We only demonstrate Property~\ref{prop:good_scenario} because the demonstration of Property~\ref{prop:bad_scenario} is similar. \\
Let $\varepsilon > 0$. For any $k \geq 1$, we consider the event $A_k$ ``the center of a complete 1-square of length $4 T_{k}$ is at a distance smaller or equal to $\varepsilon T_{k}$ from the origin''. For this event to be realized, it is sufficient that events $B_k$ and $C_k$ are realized, where
    \begin{itemize}
        \item $B_k$ is the event ``at least one of the points located at a distance less than or equal to $\varepsilon T_k$ from the origin is the center of a 1-square of length $4T_k$''.
        \item $C_k$ is the event ``There is no 0-square  of length greater than $4 T_k$  in the ball of radius $(4 + \varepsilon)T_k$ centered at the origin.''
    \end{itemize}
    Denote $\lfloor x \rfloor$ the integer part of $x$. 
We decide if a point was the center of a 1-square according to a Bernoulli of parameter $T_k^{-3}$, so 
    \begin{equation*}
        \m{P}(B_k) \geq 1 - (1 - T_k^{-3})^{(\lfloor \varepsilon T_k \rfloor^3 )} 
    \end{equation*}
As we have
    \begin{equation*}
        \lim_{k \to \infty} \left[ 1 - (1 - T_k^{-3})^{(\lfloor \varepsilon T_k \rfloor^3 )} \right] > 0, 
    \end{equation*}
we conclude that
    \begin{equation*}
    \liminf_{k \to \infty} \m{P}(B_k) > 0.
    \end{equation*}
Moreover, we have
    \begin{equation}
         \m{P}(C_k) \geq \prod_{k' \geq k+1} (1 - T_{k'}^{-3})^{[4 T_{k'} + (4 + \varepsilon)T_k]^2\cdot[(4+\varepsilon) T_k] }.
    \end{equation}
Since $T_k=2^k$, we have
    \begin{equation}
        \liminf_{k \to +\infty} \left(  \prod_{k' \geq k+1} (1 - T_{k'}^{-3})^{[4 T_{k'} + (4 + \varepsilon)T_k]^2\cdot[(4+\varepsilon) T_k] }  \right) > 0
    \end{equation}
Therefore
    \begin{equation}
        \liminf_{k \to +\infty}  \m{P}(A_k) \geq \liminf_{k \to +\infty}  \m{P}(B_k \cap C_k ) = \liminf_{k \to +\infty}   \m{P}( B_k )  \cdot \m{P}(C_k ) > 0.
    \end{equation}
Moreover, 
    \begin{equation*}
        \m{P}( \limsup_{k \to +\infty} A_k) \geq  \liminf_{k \to +\infty}  \m{P}(A_k) > 0.
    \end{equation*}
Since the law of $g_\omega$ is stationary and ergodic and $\limsup_{k \to +\infty} A_k$ is invariant by space translation, we conclude that 
    \begin{equation*}
          \m{P}( \limsup_{k \to +\infty} A_k) = 1,
    \end{equation*}
    which proves Property 1. 
\end{proof}
We can now deduce the following proposition:
\begin{proposition} \label{prop:discrete_CE}
For any $z \in \m{Z}^3$, the following  properties hold $\m{P}$-almost surely: 
\begin{equation*}
\limsup_{n \rightarrow+\infty} v^\omega_n(z) = 1 \quad \text{and} \quad 
\liminf_{n \rightarrow+\infty} v^\omega_n(z) =0. 
\end{equation*}
\end{proposition}
\begin{proof}
Let $k \geq 1$. 
Let $\omega \in \Omega$ and $\ell \geq 1$ such that there exists a 1-square of length $4T_\ell$
at distance smaller or equal to $\varepsilon T_\ell$ from the origin. Irrespective of Player 2'strategy, Player 1 can bring the state to the 1-square in less than $\varepsilon T_\ell+1$ stages. Then, by playing 0, she ensures that the state remains in the 1-square until the end of the game: indeed, the square has length $4T_\ell$, and the duration of the game is $T_\ell$. Thus, the value of the $T_\ell$-stage game is larger than $(1-\varepsilon-1/T_{\ell})$. Since Property 1 holds $\m{P}$-almost surely, this yields
$\limsup_{n \rightarrow+\infty} v^\omega_n(0) \geq 1-\varepsilon$. Because $\varepsilon$ is arbitrary and the payoff function is majorized by 1, it follows that $\limsup_{n \rightarrow+\infty} v^\omega_n(0)=1$ $\m{P}$-almost surely. Since the distribution of $(g_\omega)$ is stationary, the same property holds for any initial state $z \in \m{Z}^3$. A similar reasoning can be made for Player 2, to prove the second equality. 
\end{proof}

 \section{Connection with homogenization of Hamilton-Jacobi equations}
 \subsection{Stochastic homogenization} \label{subsection: SH}
  Hamilton-Jacobi equations is a class of Partial Differential Equations that is widely used to describe Hamiltonian systems arising in mechanics. A central research area is to study limit properties of solutions when the time scale and space scale become larger and larger: 
  when solutions converge in some sense, we say that the PDE system \textit{homogenizes}. 

Formally, let $d \geq 1$ and $H : \m{R}^d \times \m{R}^d \to \m{R}$ be a continuous function called \textit{Hamiltonian}. Assume that $H$ is periodic with respect to $x$, and \textit{coercive}, that is, $H(p,x) \rightarrow +\infty$ 
when $\left\|p\right\| \rightarrow +\infty$, uniformly in $x$.
Consider the equation 
\begin{equation} \label{eq:system_SHD}
\left\{ \begin{array}{ll}
            \partial_t u^{\varepsilon} (t,x) + H\left( \nabla u^{\varepsilon}(t,x), \frac{x}{\varepsilon}\right)  = 0  & \; \text{in} \; (0, + \infty ) \times \m{R}^d \\ 
            u^{\varepsilon}(0,x)  = u_0(x) & \; \text{in} \; \m{R}^d 
          \end{array} \right.
\end{equation}
where $\partial_t$ designates the derivative with respect to $t$, $\nabla$ the derivative with respect to $x$, and $u_0$ is a bounded and uniformly continuous function. 
Lions, Papanicolaou and Varadhan \cite{LPV86} have proved that as $\varepsilon$ vanishes, the viscosity solution $u^\varepsilon$ converges locally uniformly in $(t,x)$ to the unique solution of a system of the form :
\begin{equation} \label{eq:sys_SHL}
\left\{ \begin{array}{ll}
            \partial_tu(t,x) + \bar{H}( \nabla u(t,x)) = 0  & \; \text{in} \; (0, + \infty ) \times \m{R}^d \\
            u(0,x)  = u_0(x) & \; \text{in} \; \m{R}^d
          \end{array} \right.
\end{equation}
where $\bar{H}$ is called the \textit{effective Hamiltonian}. 
To get a better understanding of what such a convergence result means, notice that when $u_0$ is linear, then a simple change of variables shows that for all $(t,x) \in (0,+\infty) \times \m{R}^d$,
\begin{equation*}
u^{\varepsilon}(t,x)=\varepsilon u^1\left(\frac{t}{\varepsilon},\frac{x}{\varepsilon}\right).
\end{equation*}
Hence, taking $\varepsilon$ to 0 corresponds to a change in time scale and space scale. To give a concrete example, when the system (\ref{eq:system_SHD}) describes \textit{fire-front propagation} \cite{peters01,williams85} (e.g., a fire in a forest), then the front shape at time $t$ can be described by the level sets of $u^1$. Hence, taking $\varepsilon$ to 0 corresponds to looking at the front shape after a very long time and at a very large scale, putting aside the transitory regime and the local space variations. 

Following this seminal work, huge efforts were developed to generalize homogenization to random Hamiltonians. In the previous example, randomness appears naturally by considering the fact that the composition of the forest (e.g. the type of trees or the forest density) may be unknown at a local scale. This can be modeled by adding randomness to the system (\ref{eq:system_SHD}). Formally, Let $(\Omega,\mathcal{F},\m{P})$ be a probability space, equipped with a family of measurable transformations $\tau_y: \Omega \rightarrow \Omega, y \in \m{R}^d$, satisfying for all $x, y \in \m{R}^d$ and $F \in \mathcal{F}$
\begin{equation*}
\m{P}(\tau_x(F))=\m{P}(F) \quad \text{and} \quad \tau_{x+y}=\tau_x \circ \tau_y.
\end{equation*}
Moreover, $\tau$ is \textit{ergodic}: if $F \in \mathcal{F}$ is such that for all $x \in \m{R}^d$, $\tau_x(F)=F$, then $\m{P}(F) \in \left\{0,1\right\}$. 

Let $H : \m{R}^d \times \m{R}^d  \times \Omega \to \m{R}$ that is measurable and \textit{stationary}, that is, for all $(p,x,y,\omega)$,
\begin{equation*}
H(p,x,\tau_y\omega)=H(p,x+y,\omega).
\end{equation*}
Informally, these assumptions can be summarized by saying that the law of $H$ is invariant by space translation, and any event that is invariant by space translation has either probability 0 or 1. 
In this case, $H$ is said to be $\textit{stationary and ergodic}$. 
 Moreover, we assume that $H$ is coercive with respect to the first variable, that is 
 $H(p,x,\omega) \rightarrow +\infty$ when $\left\|p\right\| \rightarrow +\infty$, uniformly in $(x,\omega)$. Equation (\ref{eq:system_SHD}) becomes 
\begin{equation} \label{eq:system_SH}
\left\{ \begin{array}{ll}
            \partial_t u^{\varepsilon} (t,x,\omega) + H\left( \nabla u^{\varepsilon}(t,x,\omega), \frac{x}{\varepsilon}, \omega\right)  = 0  & \; \text{in} \; (0, + \infty ) \times \m{R}^d \\ 
            u^{\varepsilon}(0,x,\omega)  = u_0(x) & \; \text{in} \; \m{R}^d 
          \end{array} \right.
\end{equation}
The space periodicity of the Hamiltonian in equation (\ref{eq:system_SHD}) has thus been replaced by the stationarity and ergodicity properties. 
 
When $H$ is convex with respect to $p$ and under standard regularity assumptions, Souganidis \cite{souganidis99} and Rezakhanlou and Tarver \cite{RT00} have proved that 
$u^\varepsilon$ converges almost surely in $\omega$ and locally uniformly in $(t,x)$ to the unique solution of a system of the same form as (\ref{eq:sys_SHL}), hence a \textit{deterministic} system: randomness disappears with the scale change. 


These results have been extended to various frameworks, still under the assumption that the Hamiltonian is convex in $p$ (see \cite{LS05,KRV06,LS10,LS03,S09,AS131,AT141}). Quantitative results about the speed of convergence have been obtained in \cite{ACS14,MN12,AC15}.

When $H$ is not assumed to be convex (\textit{non-convex case}), only very specific frameworks were treated: for level-set convex Hamiltonians \cite{AS132}, when the law of $H$ is invariant by rotation \cite{fehrman14}, in one-dimension \cite{ATY152,G15,DK20,yilmaz20}, when the law of $H$ satisfies a \textit{finite range condition} and $H$ is positively homogenous or star-shaped \cite{AC15b,FS17}. 
The second author \cite{Z17} proved recently that homogenization does not hold in the general non-convex case. This work builds on the connection between Hamilton-Jacobi and differential games, that we detail below.

\subsection{Connection with games} \label{sec:SH_to_games}

Non-convex Hamilton-Jacobi equations can be associated with a class of continuous-time zero-sum games called \textit{zero-sum differential games}. For conciseness purpose, we will avoid technical aspects of the model, such as definition of strategies, and regularity assumptions.

Let $(\Omega,\mathcal{F},\m{P})$ be a probability space. Let $A$ and $B$ be two sets called \textit{control sets}. To each $\omega \in \Omega$, we associate a \textit{payoff function} $g_\omega: \m{R}^d \times A \times B \rightarrow \m{R}$. Let $f: \m{R}^d \times A \times B \rightarrow \m{R}^d$ be a \textit{dynamics}, and $g_0:\m{R}^d \rightarrow \m{R}$  be a \textit{terminal payoff function}. Each $\omega \in \Omega$ determines a \textit{zero-sum differential game}, where at each time $t$, Players 1 and 2 choose controls $a(t) \in A$ and $b(t) \in B$, knowing $\omega$ and the past history. The \textit{state variable} $x(t)$ follows the ordinary differential equation
\begin{equation*}
\dot{x}(t)=f(x(t),a(t),b(t)).
\end{equation*}
In the game with duration $t$, Player 1 (resp., 2) aims at maximizing (resp., minimizing) the total payoff
\begin{equation*}
 \int_0^t g_\omega(x(s),a(s),b(s)) ds+g_0(x(t)). 
\end{equation*}
In differential game literature, the usual convention is that Player 1 minimizes a \textit{cost function}, while Player 2 maximizes it. In discrete-time games, the usual convention is inverse. In order to have a unified approach, in this manuscript the maximizing player is always Player 1, and the minimizing player is always Player 2. 
  
Note that $\omega$ plays the role of a parameter, that is drawn from $\m{P}$ at the outset of the game, and announced to players. Thus, we have defined a collection of zero-sum differential games, indexed by $\omega$. 
\\
Under the standard Isaacs assumptions, 
this game has a value, that is, the maxmin over strategies of the total payoff is equal to the minmax. As in the discrete-time case, the value can be interpreted as Player 1's payoff (or Player 2's opposite payoff) when they play rationally. The value $u(t,x,\omega)$ of the game with duration $t$ and starting from $x \in \m{R}^d$ satisfies the Hamilton-Jacobi equation
\begin{equation*}
\left\{ \begin{array}{ll}
            \partial_t u (t,x,\omega) + H\left( \nabla u(t,x,\omega),x,\omega\right)  = 0  & \; \text{in} \; (0,+\infty) \times \m{R}^d \\
            u(0,x,\omega)  = g_0(x) & \; \text{in} \; \m{R}^d
          \end{array} \right.
\end{equation*}
where 
\begin{equation} \label{eq:diff_to_HJB}
H(p,x,\omega):=-\sup_{a \in A} \inf_{b \in B} \left\{ g_\omega(x,a,b)+f(x,a,b)\cdot p\right\} \, .
\end{equation}
This corresponds to the same system as (\ref{eq:system_SH}), for $\varepsilon=1$. Hence, the same change of variables as before gives the relation
$u^\varepsilon(x,t,\omega)=\varepsilon u(x/\varepsilon,t/\varepsilon,\omega)$. 
Thus, setting $x=0$, $t=1$, and $T:=1/\varepsilon$, stochastic homogenization implies that $(1/T) u(0,T,\omega)$ converges as $T$ tends to infinity: the differential game starting from 0 has a \textit{limit value}. In fact, such a convergence property for $x=0$ and $t=1$ is most of time a big step in proving stochastic homogenization. Moreover, considering Hamilton-Jacobi equations that come from a differential game is essentially without loss of generality, thanks to the representation formulas of Evans and Souganidis \cite{ES84}. Hence, understanding existence of limit value in the class of differential games defined in this section will certainly infer new results on stochastic homogenization. 
The interest of considering Hamilton-Jacobi equations through the prism of games is that instead of looking merely at the PDE, one investigates the \textit{optimal trajectories} that generate this PDE, in the spirit of weak KAM theory. 
Such a \textit{strategic approach} can be translated in a discrete-time setting, where strategies are less technical to manipulate, and literature on limit value is more abundant.
In fact, the counterexample of the second author \cite{Z17} is a perfect illustration of this method, since it is based on a differential game, that is itself inspired from a discrete-time game. 
\subsection{Comparison with literature and perspectives}
This paper is written in a game framework, hence one can not compare its content directly with the results obtained in a non-convex Hamilton-Jacobi setting that are cited in Subsection \ref{subsection: SH}. Nonetheless, thanks to the previous subsection, it is possible to make the following analogies. 
\paragraph{$H$ is coercive}
It can be seen as the case where Player 2 ``controls" the dynamics, meaning that she can bring the state wherever she wants in linear time. Indeed, to make sure that $f(x,a,b) \cdot p$ is small when $\left\|p\right\|$ is large, Player 2 has to match the positive (resp., negative) components of $f(x,a,b)$ with the negative (resp., positive) components of $p$, which is typically a controllability assumption. 
\paragraph{$H$ is convex with respect to $p$} This can be seen as the 1-Player case: indeed, the Hamiltonian in (\ref{eq:diff_to_HJB}) is then minus an infimum over linear functions. Recall that for convex and coercive Hamiltonian, Souganidis \cite{souganidis99} and Rezakhanlou and Tarver \cite{RT00} have proved homogenization.
 \paragraph{$H$ is positively homogenous with respect to $p$, or its level sets are star-shaped with respect to $p=0$} In addition to coercivity, these are respectively the assumptions in \cite{AC15b} and \cite{FS17}. Under these assumptions, the problem can be reduced to a class of equations describing games where Player 2 aims at reaching an hyperplane in minimal time. 
 
 With these interpretations at hand, let us compare the results of this article with the homogenization literature. Notice that no controllability assumption is made in this paper ; hence, adapting the results to a PDE framework can yield homogenization results without coercivity assumption. This would be quite interesting since, even in the periodic case, homogenization for noncoercive Hamiltonians is not well understood \cite{cardaliaguet10}. To this aim, a crucial point will be to translate the ``oriented game assumption" to a PDE framework. To give an example of the type of Hamiltonians that are expected to satisfy the ``translated" assumption, consider in $\m{R}^2$ an Hamiltonian of the form
\begin{equation*}
H(p_1,p_2,x,\omega)=-\sup_{a \in A} \inf_{b \in B} \left\{g_\omega(x,a,b)+p_1+f(x,a,b) \cdot p_2 \right\}.
\end{equation*}
Note that this Hamiltonian is neither star-shaped, nor positively homogenous. 
\\

Going back to the existence of limit value in the discrete-time game framework, the following questions are of particular interest: 
\begin{enumerate}
\item
Does limit value exist for games that are not oriented, with or without controllability assumption? The following particular case in $\m{Z}^2$ can be considered as a benchmark: 
Player 1 has two actions, ``go up or go down", and Player 2 has two actions ``go left or go right". This describes in an obvious way the transitions of the game. As for the payoff, given some fixed parameter $p \in [0,1]$, in each state it is equal to 1 with probability $p$, and 0 otherwise.  
\item
Does limit value exist in the one-player setting?
\item
Does limit value exist when payoffs are not i.i.d., but correlated in a weak sense?
\item
Does limit value exist in 2-dimensional oriented percolation games?  
\bibliography{bibliogen.bib}
\end{enumerate}
\end{document}